\documentclass[14pt]{amsart}
\usepackage{amssymb}
\usepackage{amsfonts}
\usepackage{amsmath}
\usepackage{mathrsfs}
\usepackage{amsthm}
\newtheorem{theorem}{Theorem}[section]
\newtheorem{lemma}[theorem]{Lemma}

\theoremstyle{definition}

\theoremstyle{remark}
\newtheorem{remark}{Remark}

\newcommand{\x}{\tilde{x}}
\newcommand{\R}{\mathbb{R}}

\newcommand{\N}{\mathbb{N}}

\newcommand{\MS}{\mathbb{S}}
\newcommand{\F}{\mathcal F}

\newcommand{\innerp}[1]{\langle {#1} \rangle}

\newcommand{\argmin}[1]{\mathop{\rm argmin}\limits_{#1}}
\newcommand{\abs}[1]{\left\lvert#1\right\rvert}

\title{The lower bound of the PCM quantization error in high dimension}\thanks{Z. Xu is supported by the National Natural Science Foundation of China (11171336 and 11331012) and by the Funds for Creative
       Research Groups of China (Grant No. 11021101).}
\author{Heng Zhou}
\address{School of Sciences, Tianjin Polytechnic University\\ Tianjin 300160, China;\newline  {\tt Email: zhouheng7598@sina.com.cn}}
\author{Zhiqiang Xu}
\address{LSEC, Institute of Computational Mathematics, Academy of Mathematics and System Sciences,
Chinese Academy of Sciences, Beijing 100190, China; {\tt Email: xuzq@lsec.cc.ac.cn}}
\date{}
\begin{document}
\maketitle
\begin{center}
\parbox{11cm}{{\it Abstract}\quad
In this note, we investigate the performance of  the PCM scheme with linear quantization rule for quantizing  unit-norm tight frame expansions for $\R^d$  without the White Noise Hypothesis. In \cite{WX}, Wang and Xu showed  that for asymptotically equidistributed unit-norm tight frame the PCM quantization error has an upper bound ${\mathcal O}(\delta^{(d+1)/2})$ and they conjecture the upper bound is sharp. In this note,
we confirm the conjecture with employing  the asymptotic estimate of the Bessel functions.}
\end{center}
\begin{center}
\vspace{0.2cm}
\parbox{11cm}{{\it Key words and phrases}\quad Unit-norm tight frames; PCM quantization error; Bessel functions; Orthogonal polynomials
}
\end{center}
\vspace{0.2cm}
\begin{center}
\parbox{11cm}{{\it AMS Subject Classification 2000}\quad
42C15; 33C10; 42A05}
\end{center}
\mbox{}\\[10pt]
\section{Introduction}

In signal processing, one of the primary goals is to  find a  digital  representation for a given signal that is suitable for storage, transmission, and recovery.
  We assume that the signal $x$ is an element of a finite-dimensional Hilbert space ${\mathbb H}=\mathbb{R}^{d}$. One often begins to expand $x$ over a dictionary $\mathcal{F}=\{e_{j}\}_{j=1}^{N}$, i.e.,
$$x=\sum_{j=1}^{N}c_{j}e_{j},$$
where $c_{j}$ are real numbers. We say $\mathcal{F}$ is a {\em tight  frame} of $\mathbb{R}^{d}$ if
 $$
 x=\frac{d}{N}\sum_{j=1}^N\innerp{x,e_j}e_j
 $$
 holds for all $x\in \R^d$. The tight frame is called {\em unit-norm} if $\|e_j\|_2=1$ holds for all $1\leq j\leq N$.
In the digital domain  the coefficients $x_{j}=\innerp{ x, e_{j}}$  must be mapped to a discrete set of values $\mathcal{A}$ which is called the quantization alphabet. The simplest way for such a mapping is the Pulse Code Modulation (PCM) quantization scheme, which has $\mathcal{A}=\delta\mathbb{Z}$ with $\delta>0$ and  the mapping is done by the function
$$
Q_{\delta}(t):=\argmin{r\in\mathcal{A}}|t-r|=\delta\left\lfloor\frac{t}{\delta}+\frac{1}{2}\right\rfloor.
$$
 Thus in practical applications we in fact have only a quantized representation
 $$
q_{j}:=Q_{\delta}(\innerp{x, e_{j}}), \qquad j=1, \ldots, N
 $$
  for each $x\in\mathbb{R}^{d}$. The linear  reconstruction is
$$
\tilde{x}_{\mathcal{F}}=\frac{d}{N}\sum_{j=1}^{N} q_{j}e_{j}.
$$
Naturally we are interested in the error for this reconstruction, i.e.
 $$
 E_{\delta}(x, \mathcal{F}):=\| x-\tilde{x}_{\mathcal{F}}\|,
 $$
  where $\|\cdot\|$ is $\ell_{2}$ norm.
  To simplify the investigation of  $E_{\delta}(x, \mathcal{F})$, one employs  the {\em White Noise Hypothesis } (WNH)  in this area (see \cite{GoVeTh98,Benn48,BePoYi06,BePoYi06II,BorWan09, JiWaWa07}), which asserts that the quantization error sequence $\{x_j-q_j\}_{j=1}^N$ can be modeled as
 an independent sequence of i.i.d. random variables that are uniformly distributed on the
  interval $(-\delta/2, \delta/2)$. Under  the WNH, one can obtain the mean square error
$$
     MSE\,\,=\,\,{\mathbb E}(\|x-\x_\F\|^2)\,\,=\,\,\frac{d^2\delta^2}{12 N}.
$$
The result implies that the MES of $E_\delta(x,{\mathcal F})$ tends to $0$ with $N$ tending to infinity. However, as pointed out in  \cite{BorWan09, JiWaWa07},  the WNH  only asymptotically holds for fine quantization  (i.e. as $\delta$
tends to 0) under rather general conditions. So, for a fixed $x$, one is  interested in whether $E_{\delta}(x, \mathcal{F})$  really tending to $0$ without WNH. The result in \cite{WX} gives a solution for the case where $d=2$ which shows that for some $x\in \R^2$ the quantization error $E_{\delta}(x, \mathcal{F})$ does not diminish to $0$ with $N$ tending to infinity.
Naturally, one would like to know whether it is possible to extend the result to higher dimension.
 In \cite{WX}, Wang and Xu  investigate the case where $\F$ is the {\em asymptotically equidistributed  unit-norm tight frame} in $\R^d$.
 A sequence of finite sets $A_m \subset \MS^{d-1}$ with cardinality
$N_m = \# A_m$ is said to be {\em asymptotically equidistributed} on $\MS^{d-1}$ if
for any piecewise continuous function $f$ on $\MS^{d-1}$ we have
$$
\lim_{m\rightarrow\infty} \frac{1}{N_m}\sum_{v\in A_m}
f(v)\,\,=\,\,\int_{{z}\in \MS^{d-1}}f({z})d\nu,
$$
where $f$ are  piecewise continuous functions on $\MS^{d-1}$ and
$d\nu$ denotes the normalized Lebesgue measure on $\MS^{d-1}$.
Then the following theorem presents an upper bound for $\lim_{m\rightarrow\infty}E_\delta(x,\F_m)$.

\begin{theorem}\cite{WX}\label{th:lowbound}
  Assume that $\F_m$ are asymptotically equidistributed  unit-norm tight frames in $\R^{d}$. Then for any $x\in \R^{d}$ we have
    $$
      \lim_{m\rightarrow\infty}E_\delta(x,\F_m)\leq C_d \frac{\delta^{(d+1)/2}}{r^{(d-1)/2}},
    $$
where $r=\|x\|$ and $C_d$ is a constant depending on $d$.
\end{theorem}
A main tool
for obtaining Theorem  \ref{th:lowbound} is Euler-Maclaurin formula. However, it seems that it is difficult to
 extend the method to obtain the lower bound.
In \cite{WX}, Wang and Xu conjecture the bound $O\left(\frac{\delta^{(d+1)/2}}{r^{(d-1)/2}}\right)$ is sharp.  In this note, we employ the tools of Bessel function and hence
confirm the conjecture. In particular, we have:

\begin{theorem}\label{th:main}
Suppose that $d>2$ is an integer. Assume that $x\in \R^d$ and that $\F_m$ are asymptotically equidistributed  unit-norm tight frames in $\R^{d}$.
Set $r:=\|x\|, R:={r}/{\delta}$, $\epsilon:=R-\lfloor R\rfloor$, and
$$
I\,\,:=\,\,d\int_{0}^{2\pi}|(\sin(\theta_{2}))^{d-3}\cdots\sin\theta_{d-2}|d\theta_{2}\cdots d\theta_{d-2}\,\,>\,\,0.
$$
\begin{enumerate}
\item[(i)]
If $d=2n$ and $1/4\leq\epsilon\leq 1/2$, then
$$\lim_{m\rightarrow\infty}E_{\delta}(x, \mathcal{F}_{m})\,\,\geq\,\, C_{1,d}\frac{\delta^{\frac{d+1}{2}}}{r^{\frac{d-1}{2}}},$$
provided that $R={\|x\|}/{\delta}$ is big enough,
where
$$C_{1, d}=\frac{(n-1)!\cdot {2n-2\choose n-1}\cdot M_{1}}{2^{2n-2}\pi^{n}}\cdot I,\quad M_{1}=\frac{4}{5}|\cos(2\pi\epsilon-\frac{3}{4}\pi)|-\frac{5}{4}\sum_{k=2}^{+\infty}\frac{1}{k^{\frac{2n+1}{2}}}>0.$$
\item[(ii)]
If $d=2n+1$ and $1/6\leq\epsilon\leq 1/3$, then
$$\lim_{m\rightarrow\infty}E_{\delta}(x, \mathcal{F}_{m})\,\,\geq\,\, C_{2,d}\frac{\delta^{\frac{d+1}{2}}}{r^{\frac{d-1}{2}}},$$
provided that $R={\|x\|}/{\delta}$ is big enough,
where
$$C_{2,d}=\frac{(n-1)!\cdot M_2}{\pi^{n+1}}\cdot I, \quad M_{2}=\frac{7}{8}|\cos(2\pi\epsilon-\frac{1}{2}\pi)|-\frac{8}{7}\sum_{k=2}^{+\infty}\frac{1}{k^{n+1}}>0.$$
\end{enumerate}
\end{theorem}

After introducing some necessary concepts and results to be used in our investigation in Section 2, we
present the proof of Theorem \ref{th:main} in Section 3.
\section{Preliminaries}

{\bf Bessel function.} (see \cite{book})
For $\alpha >0$, the {\em Bessel function} $J_\alpha$ is defined by the series representation
\begin{equation}\label{eq:le22}
J_{\alpha}(x)=\sum_{k=0}^{+\infty}\frac{(-1)^{k}}{k!\cdot \Gamma(k+\alpha+1)}\left(\frac{x}{2}\right)^{2k+\alpha}.
\end{equation}
Particularly, when  $\alpha \in{\mathbb{N}}$, we have
$$
J_\alpha(x)=\frac{1}{\pi}\int_0^\pi \cos(\alpha \tau-x\sin(\tau))d\tau.
$$
We also need an asymptotic estimate for the Bessel function $J_{\alpha}$ which is presented in \cite{Bessel}.
\begin{theorem}{\rm (\cite{Bessel})}\label{le:Jvx}
For the Bessel function $J_{\alpha}$, we have
$$J_{\alpha}(x)=\sqrt{\frac{2}{\pi x}}\cos(x-\omega_{\alpha})+\theta c\mu x^{-\frac{3}{2}},$$
where $\omega_{\alpha}=\frac{\pi\alpha}{2}+\frac{1}{4}\pi$, $\mu=|\alpha^{2}-\frac{1}{4}|$, $|\theta|\leq 1$, and
$$c=\left\{
  \begin{array}{ll}
    ({2}/{\pi})^{{3}/{2}}, & \quad x\geq 0, \quad |\alpha|\leq \frac{1}{2}\\
    {\sqrt{2}}/{2}, & \quad x\geq \sqrt{\mu}, \quad \alpha>\frac{1}{2}\\
    {5}/{4},& \quad 0<x<\sqrt{\mu}, \quad \alpha>\frac{1}{2}.\\
    \end{array}\right.
$$
\end{theorem}
{\bf Combinatorics identity.}
(see (7.7) of Table 4 in \cite{Gould})
\begin{equation}\label{eq:gould}
\sum_{m=0}^{h}(-1)^{m}{n+h\choose h-m}{n+h\choose h+m}=\frac{1}{2}{n+h\choose h}+\frac{1}{2}{n+h\choose h}^2
\end{equation}
\section{Proof of Theorem \ref{th:main}}
To this end, we first introduce several lemmas:
\begin{lemma}\label{le:31}
For all $n\in \mathbb{N^{+}}$ and $h\in \mathbb{N}$ we have
\begin{eqnarray}
\sum_{m=0}^{h}(-1)^{m}(2m+1)
{n+h\choose h-m}
{n+h\choose h+m+1}
=n{n+h\choose n},\label{eq:le11}
\end{eqnarray}
and
\begin{eqnarray}
\sum_{m=l}^{h}(-1)^{m}(2m+1)
{2h+1\choose h-m}
{m+l\choose 2l}=0,\quad l=0,1,\ldots,h-1.\label{eq:lell1}
\end{eqnarray}
\end{lemma}
\begin{proof} We prove (\ref{eq:le11}) by induction. To state conveniently, set
$$
A_n^h:=\sum_{m=0}^{h}(-1)^{m}(2m+1){n+h\choose h-m}{n+h\choose h+m+1}.
$$
A simple observation is that (\ref{eq:le11}) holds when  $n\in \N^+, h=0$ and when $n=0, h\in \N^+$.
Assume that $n_0, h_0\in \N^+$. For the induction step, we assume that (\ref{eq:le11}) is true both for  $n\leq n_0\in\mathbb{N^{+}}, h\in \N^+$ and for $n=n_0+1, h\leq h_0-1\in \N^+$.
To this end, we just need prove that the result holds for  $n=n_0+1, h=h_0$.
We have
{\small
\begin{equation*}
\begin{aligned}
A_{n_0+1}^{h_0}
&=\sum_{m=0}^{h_0}(-1)^{m}(2m+1)
  {n_0+h_0+1\choose h_0-m} {n_0+h_0+1\choose h_0+m+1}\\
&=\sum_{m=0}^{h_0}(-1)^{m}(2m+1)\left({n_0+h_0\choose h_0-m}+{n_0+h_0\choose h_0-m-1}\right)\left(
{n_0+h_0\choose h_0+m+1}+{n_0+h_0\choose h_0+m}\right)\\
&=A_{n_0}^{h_0}+A_{n_0+1}^{h_0-1}+\sum_{m=0}^{h_0}(-1)^{m}(2m+1)\left({n_0+h_0\choose h_0-m}{n_0+h_0\choose h_0+m}+
{n_0+h_0\choose h_0-m-1}{n_0+h_0\choose h_0+m+1}\right)\\
&=A_{n_0}^{h_0}+A_{n_0+1}^{h_0-1}+2\sum_{m=0}^{h_0}(-1)^{m}
{n_0+h_0\choose h_0-m}{n_0+h_0\choose h_0+m}
-{n_0+h_0\choose h_0}^2\\
&=A_{n_0}^{h_0}+A_{n_0+1}^{h_0-1}+{n_0+h_0\choose n_0}=(n_0+1){n_0+h_0+1\choose n_0+1},
\end{aligned}
\end{equation*}
}
where the last equality uses the identity (\ref{eq:gould}) and the induction assumption.

We now turn to  (\ref{eq:lell1}). Set
\begin{equation}\label{eq:gm}
g_{m}:=(-1)^{m+1}(h+m+1)(m-l)\frac{{2h+1\choose h-m}{m+l\choose 2l}}{h-l}.
\end{equation}
A simple calculation shows that

\begin{equation}\label{eq:gosper}
\begin{aligned}
g_{m+1}-g_{m} =& (-1)^{m}\frac{{2h+1\choose h-m}{m+l\choose 2l}}{h-l}\left((h-m)(l+m+1)+(h+m+1)(m-l)\right)\\
=& (-1)^{m}(2m+1){2h+1\choose h-m}{m+l\choose 2l}.
\end{aligned}
\end{equation}
Then
\begin{eqnarray*}
& &\sum_{m=l}^{h}(-1)^{m}(2m+1)
{2h+1\choose h-m}
{m+l\choose 2l}\\
& &=
\sum_{m\geq l}(-1)^{m}(2m+1)
{2h+1\choose h-m}
{m+l\choose 2l}\\
& &=\sum_{m\geq l}g_{m+1}-g_{m}=g_{l}=0.\nonumber
\end{eqnarray*}
Here, the first equality holds since ${n\choose k}=0$ provided $k<0$.
\end{proof}
\begin{remark}
A key step to prove (\ref{eq:lell1}) is to construct the sequence $g_m$ which satisfies (\ref{eq:gosper}).
In the proof of Lemma \ref{le:31}, we  obtain $g_m$  using Gosper algorithm \cite{gosper}. However,
it is also simple to verify (\ref{eq:gosper}) by hand.
\end{remark}
We introduce the following results for Bessel functions
\begin{lemma}\label{le:2.2}
Set
\begin{eqnarray*}
L_{m}&:=&\int_{-\pi}^{\pi}\cos(2m+1)\theta\cos\theta(\sin\theta)^{2n-2}d\theta \\
D_{m}&:=&\int_{0}^{\pi}\cos(2m+1)\theta\cos\theta(\sin\theta)^{2n-1}d\theta .
\end{eqnarray*}
Then we have
\begin{eqnarray}
\sum_{m=0}^{n-1}(-1)^{m}L_{m}J_{2m+1}(x)&=&L_{0}2^{n-1}n!\frac{1}{x^{n-1}}J_{n}(x), \label{eq:le21}\\
\sum_{m=0}^{+\infty}(-1)^{m}D_{m}J_{2m+1}(x)&=&\sqrt{\pi}2^{n-\frac{3}{2}}(n-1)!\frac{1}{x^{n-\frac{1}{2}}}J_{n+\frac{1}{2}}(x).\label{eq:le211}
\end{eqnarray}
\end{lemma}
\begin{proof}
To this end, we first calculate  the value of $L_m$.
Using the expansion
 $$
 \sin^{2n-2}\theta=\frac{1}{2^{2n-2}}{2n-2\choose n-1}+\frac{2}{2^{2n-2}}\sum_{k=0}^{n-2}(-1)^{n-1-k}{2n-2\choose k}\cos((2n-2-2k)\theta),
 $$
 we can obtain that
 \begin{equation}\label{eq:LM}
 L_m=(-1)^m \frac{\pi}{2^{2n-2}} {2n-2\choose n+m-1} \frac{2m+1}{n+m}.
 \end{equation}
Recall that  the series representation of the Bessel function $J_{\alpha}$
\begin{equation}\label{eq:le22}
J_{\alpha}(x)=\sum_{k=0}^{+\infty}\frac{(-1)^{k}}{k!\Gamma(k+\alpha+1)}\left(\frac{x}{2}\right)^{2k+\alpha}.
\end{equation}
Substituting   (\ref{eq:le22}) into (\ref{eq:le21}) we obtain that
\begin{equation}\label{eq:le286}
\sum_{k=0}^{+\infty}\frac{(-1)^{k}}{k!}\sum_{m=0}^{n-1}(-1)^{m}L_{m}\frac{1}{\Gamma(2m+k+2)}\left(\frac{x}{2}\right)^{2k+2m+1}
=L_{0}\cdot n!\cdot \sum_{k=0}^{+\infty}\frac{(-1)^{k}}{k!\cdot \Gamma(k+n+1)}\left(\frac{x}{2}\right)^{2k+1}.
\end{equation}
To this end, we just need prove (\ref{eq:le286}).
Comparing the coefficients of the powers of $x$ on the both sides of (\ref{eq:le286}),
we only need  prove
\begin{eqnarray}
\sum_{m=0}^{h}L_{m}\frac{1}{(h-m)!\cdot \Gamma(h+m+2)}=L_{0}\cdot n!\cdot \frac{1}{h!\cdot \Gamma(h+n+1)}, \label{eq:le231}
\end{eqnarray}
which is equivalent to
\begin{equation}\label{eq:finaleq}
\sum_{m=0}^{h}(-1)^{m}(2m+1)
{n+h\choose h-m}
{n+h\choose h+m+1}
=n{n+h\choose n}.
\end{equation}
Here, we use (\ref{eq:LM}).
According to  Lemma \ref{le:31},   (\ref{eq:finaleq}) holds which in turn implies  (\ref{eq:le21}).

We next turn to  (\ref{eq:le211}). Substitute (\ref{eq:le22}) into (\ref{eq:le211}) and
compare the coefficients of the powers of $x$ on the two sides of this equation,
we only need to prove
\begin{eqnarray}
\sum_{m=0}^{h}\frac{D_{m}}{(h-m)!\cdot (h+m+1)!}=\frac{(n-1)!}{4}\cdot \frac{2^{2h+2n+3}\cdot (h+n+1)!}{h!\cdot (2h+2n+2)!},\,\, h=0,1,\ldots.\label{eq:le2222}
\end{eqnarray}
Using
\begin{equation*}
\cos nx=\frac{n}{2}\sum_{k=0}^{\lfloor\frac{n}{2}\rfloor}(-1)^{k}{n-k\choose k}\frac{(2\cos x)^{n-2k}}{n-k},\nonumber
\end{equation*}
and
\begin{equation*}
\int_{0}^{\frac{\pi}{2}}(\sin t)^{x}(\cos t)^{y}dt=\frac{\pi}{2^{x+y+1}}\frac{x!\cdot y!}{(\frac{x}{2})!\cdot (\frac{y}{2})!\cdot (\frac{x+y}{2})!},\nonumber
\end{equation*}
where $x!=\Gamma(x+1)$ for $x>0$, we have
\begin{equation}\label{eq:le2223}
D_{m}=(2m+1)\sum_{k=0}^{m}(-1)^{k}{2m+1-k\choose k}\frac{\sqrt{\pi}(n-1)!}{4(2m+1-k)}\frac{(2m+2-2k)!}{(m+1-k)!\cdot (\frac{2m+2n-2k+1}{2})!}.
\end{equation}
Substituting (\ref{eq:le2223}) into (\ref{eq:le2222}), we can rewrite  (\ref{eq:le2222}) as
\begin{equation}\label{eq:le2224}
\begin{aligned}
& \sum_{m=0}^{h}\frac{2m+1}{(h-m)!\cdot(h+m+1)!}\sum_{k=0}^{m}(-1)^{k}{2m+1-k\choose k}\frac{1}{2m+1-k}\frac{(2m+2-2k)!}{(m+1-k)!\cdot (\frac{2m+2n-2k+1}{2})!}\\
&  =\frac{1}{\sqrt{\pi}}\frac{2^{2h+2n+3}(h+n+1)!}{h!\cdot (2h+2n+2)!}.
\end{aligned}
\end{equation}
On the other hand, we can rewrite the left side of (\ref{eq:le2224}) as
\begin{equation}\label{eq:le2225}
\begin{aligned}
& \sum_{m=0}^{h}\frac{2m+1}{(h-m)!\cdot (h+m+1)!}\sum_{k=0}^{m}(-1)^{k}{2m+1-k\choose k}\frac{1}{2m+1-k}\frac{(2m+2-2k)!}{(m+1-k)!\cdot (\frac{2m+2n-2k+1}{2})!}\\
&  =\sum_{m=0}^{h}\frac{2m+1}{(h-m)!\cdot (h+m+1)!}\sum_{l=0}^{m}(-1)^{m-l}{m+l+1\choose m-l}\frac{1}{m+l+1}\frac{(2l+2)!}{(l+1)!\cdot (\frac{2l+2n+1}{2})!}\\
&  =\sum_{l=0}^{h}\frac{1}{(\frac{2l+2n+1}{2})!}\sum_{m=l}^{h}\frac{2m+1}{(h-m)!\cdot (h+m+1)!}(-1)^{m-l}{m+l+1\choose m-l}\frac{1}{m+l+1}\frac{(2l+2)!}{(l+1)!}.
\end{aligned}
\end{equation}
Here, in the first equality, we set  a  new variable  $l:=m-k$.
To this end, we consider the second term on the right side of the last equality in (\ref{eq:le2225}). Note that, for $l=0,\ldots,h-1$,
\begin{eqnarray*}
& &\sum_{m=l}^{h}\frac{2m+1}{(h-m)!\cdot (h+m+1)!}(-1)^{m-l}{m+l+1\choose m-l}\frac{1}{m+l+1}\frac{(2l+2)!}{(l+1)!}\\
& & =\frac{(-1)^{l}(2l+2)!}{(2l+1)(2h+1)!\cdot (l+1)!}\sum_{m=l}^{h}(2m+1){2h+1\choose h-m}(-1)^{m}{m+l\choose 2l}\\
& &=0.
\end{eqnarray*}
Here, the last equality follows from (\ref{eq:lell1}) in Lemma 3.1.
Hence the last summation in (\ref{eq:le2225})  is reduced to
\begin{eqnarray}
& &\frac{1}{(\frac{2h+2n+1}{2})!}\frac{2h+1}{(h+h+1)!}(-1)^{h-h}{h+h+1\choose h-h}\frac{1}{h+h+1}\frac{(2h+2)!}{(h+1)!}\nonumber\\
& &=\frac{1}{(\frac{2h+2n+1}{2})!}\frac{2}{h!}
=\frac{1}{\sqrt{\pi}}\frac{2^{2h+2n+3}(h+n+1)!}{h!\cdot (2h+2n+2)!}.\nonumber
\end{eqnarray}
Here, the last equality uses
$$
\left(\frac{2h+2n+1}{2}\right)!=\Gamma\left(h+n+1+\frac{1}{2}\right)=\frac{(2h+2n+2)!}{2^{2h+2n+2}(h+n+1)!}\sqrt{\pi}.
$$
 We arrive at the conclusion.

\end{proof}
Now we can give an estimation  for the integrals $\int_{0}^{\pi}\Delta_{\delta}(r\cos\theta)\cos\theta(\sin\theta)^{2n-2}d\theta$
and $\int_{0}^{\pi}\Delta_{\delta}(r\cos\theta)\cos\theta(\sin\theta)^{2n-1}d\theta$.\\
\begin{lemma}\label{le:3.3}
Set $R:=\frac{r}{\delta}$ and $\epsilon:=R-\lfloor R\rfloor$.
Then when $1/4\leq \epsilon\leq 1/2$
\begin{equation}\label{eq:thlowbup}
\begin{aligned}
\frac{(n-1)!\cdot {2n-2\choose n-1}M_{1}}{2^{2n-2}\cdot \pi^{n}}\frac{\delta^{\frac{2n+1}{2}}}{r^{\frac{2n-1}{2}}} &\leq
\left|\int_{0}^{\pi}\Delta_{\delta}(r\cos\theta)\cos\theta(\sin\theta)^{2n-2} d\theta\right|\\ &\leq
\frac{5}{4}\frac{(n-1)!\cdot {2n-2\choose n-1}
 \sum_{k=1}^{+\infty}\frac{1}{k^{\frac{2n+1}{2}}}}{2^{2n-2}\cdot \pi^{n}}\frac{\delta^{\frac{2n+1}{2}}}{r^{\frac{2n-1}{2}}},
 \end{aligned}
\end{equation}
provided that $R=\frac{r}{\delta}$ is big enough, where $M_{1}=\frac{4}{5}|\cos(2\pi\epsilon-\frac{3}{4}\pi)|-\frac{5}{4}\sum_{k=2}^{+\infty}\frac{1}{k^{\frac{2n+1}{2}}}>0$ and $n\geq 2$. When $1/6 \leq \epsilon\leq 1/3$,  we have
\begin{equation}\label{eq:thupbound}
\begin{aligned}
\frac{M_{2}(n-1)!}{\pi^{n+1}}\frac{\delta^{n+1}}{r^{n}} &\leq
\left|\int_{0}^{\pi}\Delta_{\delta}(r\cos\theta)\cos\theta(\sin\theta)^{2n-1} d\theta\right|\\ &\leq
\frac{8}{7}\frac{(n-1)!\sum_{k=1}^{+\infty}\frac{1}{k^{n+1}}}{\pi^{n+1}}\frac{\delta^{n+1}}{r^{n}},
\end{aligned}
\end{equation}
provided that $R=\frac{r}{\delta}$ is big enough, where $M_{2}=\frac{7}{8}|\cos(2\pi\epsilon-\frac{1}{2}\pi)|-\frac{8}{7}\sum_{k=2}^{+\infty}\frac{1}{k^{n+1}}>0$ and $n\geq 1$.
\end{lemma}

\begin{proof}
Firstly we consider  $\int_{0}^{\pi}\Delta_{\delta}(r\cos\theta)\cos\theta(\sin\theta)^{2n-2}d\theta$.
Using the Fourier expansion for $\lfloor x\rfloor$ with $x\in\mathbb{R}\setminus\mathbb{Z}$,
$$\lfloor x\rfloor=x-\frac{1}{2}+\frac{1}{\pi}\sum_{k=1}^{+\infty}\frac{\sin(2k\pi x)}{k},$$
we have
\begin{equation}\label{eq:zhuanbian}
\begin{aligned}
&\int_{0}^{\pi}\Delta_{\delta}(r\cos\theta)\cos\theta(\sin\theta)^{2n-2}d\theta\\
 &=
\frac{1}{2}\int_{-\pi}^{\pi}\Delta_{\delta}(r\cos\theta)\cos\theta(\sin\theta)^{2n-2}d\theta\\ &=
-\frac{\delta}{2\pi}\int_{-\pi}^{\pi}
\sum_{k=1}^{+\infty}\frac{\sin(2k\pi\frac{r\cos\theta}{\delta}+k\pi)}{k}\cos\theta(\sin\theta)^{2n-2}d\theta\\
 &= -\frac{\delta}{2\pi}
\sum_{k=1}^{+\infty}\frac{(-1)^{k}}{k}\int_{-\pi}^{\pi}\sin(2k\pi\frac{r}{\delta}\cos\theta)\cos\theta(\sin\theta)^{2n-2}d\theta\\
&= -\frac{\delta}{\pi}
\sum_{k=1}^{+\infty}\frac{(-1)^{k}}{k}\sum_{m=0}^{n-1}(-1)^{m}L_{m}J_{2m+1}(2k\pi\frac{r}{\delta})\\
&= -\frac{\delta}{\pi}\cdot L_{0}\cdot 2^{n-1}n!\sum_{k=1}^{+\infty}\frac{(-1)^{k}}{k}
\frac{1}{(2k\pi\frac{r}{\delta})^{n-1}}\cdot J_{n}(2k\pi\frac{r}{\delta})\\
&=-\frac{1}{\pi^n}\cdot \frac{\delta^n}{r^{n-1}}\cdot L_0\cdot n!\cdot \sum_{k=1}^\infty \frac{(-1)^k}{k^n}J_n(2k\pi \frac{r}{\delta}) \\
&=-\frac{1}{\pi^n}\cdot \frac{\delta^n}{r^{n-1}}\cdot \frac{\pi}{2^{2n-2}}\cdot {2n-2\choose n-1} \cdot (n-1)!\cdot \sum_{k=1}^\infty \frac{(-1)^k}{k^n}J_n(2k\pi \frac{r}{\delta}).
\end{aligned}
\end{equation}
In the fourth equality, we use the formula
$$\sin(x\cos\theta)=2\sum_{m=0}^{+\infty}(-1)^m\cos((2m+1)\theta) J_{2m+1}(x)$$
and the orthogonality of the systems $\{\cos kx\}_{k=0}^{+\infty}$ on the interval $[-\pi, \pi]$. We use (\ref{eq:le21})  in the fifth  equality.
To this end, according to (\ref{eq:zhuanbian}), we only need to estimate
$$
\left|\sum_{k=1}^{+\infty}\frac{(-1)^{k}}{k^{n}}J_{n}(2k\pi\frac{r}{\delta})\right|.
$$
In fact, note that
\begin{equation}\label{eq:zhuanbian1}
|J_{n}(2\pi\frac{r}{\delta})|-\sum_{k=2}^{+\infty}\frac{1}{k^{n}}
|J_{n}(2k\pi\frac{r}{\delta})|\,\,\leq\,\, \left|\sum_{k=1}^{+\infty}\frac{(-1)^{k}}{k^{n}}
J_{n}(2k\pi\frac{r}{\delta})\right|\,\,\leq\,\,\sum_{k=1}^{+\infty}\frac{1}{k^{n}}
|J_{n}(2k\pi\frac{r}{\delta})|.
\end{equation}
We first consider $|J_{n}(2\pi\frac{r}{\delta})|-\sum_{k=2}^{+\infty}\frac{1}{k^{n}}
|J_{n}(2k\pi\frac{r}{\delta})|$.
Using the asymptotic estimate for $J_{n}(x)$ in Theorem \ref{le:Jvx},
we have
\begin{equation}\label{eq:jnexp}
\begin{aligned}
\abs{J_n(2\pi \frac{r}{\delta})}&=\abs{\frac{1}{\pi }\sqrt{\frac{\delta}{r}}\cos(2\pi \frac{r}{\delta}-\omega_n)+\theta c \mu (2\pi \frac{r}{\delta})^{-3/2}}\\
&=\abs{\frac{1}{\pi }\sqrt{\frac{\delta}{r}}\cos(2\pi \epsilon-\frac{(2n+1)\pi}{4})+\theta c \mu (\frac{1}{2\pi} \frac{\delta}{r})^{3/2}}
\end{aligned}
\end{equation}
which implies that
\begin{eqnarray*}
\abs{J_n(2\pi \frac{r}{\delta})}\geq \frac{4}{5}\abs{\frac{1}{\pi }\sqrt{\frac{\delta}{r}}\cos(2\pi \epsilon-\frac{3\pi}{4})}
\end{eqnarray*}
provided that $R=\frac{r}{\delta}$ is big enough.
On the other hand,
\begin{equation}\label{eq:onthe}
\begin{aligned}
\sum_{k=2}^{+\infty}\frac{1}{k^{n}}
|J_{n}(2k\pi\frac{r}{\delta})|=&\sum_{k=2}^{+\infty}\frac{1}{k^{n}}\abs{\frac{1}{\pi }\sqrt{\frac{\delta}{kr}}\cos(2\pi k\frac{r}{\delta}-\omega_n)+\theta c \mu (2\pi k\frac{r}{\delta})^{-3/2}}\\
=&\sum_{k=2}^{+\infty}\frac{1}{k^{n}}\abs{\frac{1}{\pi }\sqrt{\frac{\delta}{kr}}\cos(2\pi k\epsilon-\frac{(2n+1)\pi}{4})+\theta c \mu (2\pi k\frac{r}{\delta})^{-3/2}}.
\end{aligned}
\end{equation}
Therefore,  according to (\ref{eq:jnexp}),
\begin{eqnarray*}
\sum_{k=2}^{+\infty}\frac{1}{k^{n}}
|J_{n}(2k\pi\frac{r}{\delta})|\,\,\leq\,\, {\frac{5}{4\pi }\sqrt{\frac{\delta}{r}}\sum_{k=2}^{+\infty}\frac{1}{k^{\frac{2n+1}{2}}}}
\end{eqnarray*}
provided that $R=\frac{r}{\delta}$ is big enough.
Combining  above results,  we obtain that
\begin{equation}\label{eq:zhuanbian2}
|J_{n}(2\pi\frac{r}{\delta})|-\sum_{k=2}^{+\infty}\frac{1}{k^{n}}
|J_{n}(2k\pi\frac{r}{\delta})|\,\,\geq\,\,\frac{1}{\pi }\sqrt{\frac{\delta}{r}}\left( \frac{4}{5}\abs{\cos(2\pi \epsilon-\frac{3\pi}{4})}-\frac{5}{4}{\sum_{k=2}^{+\infty}\frac{1}{k^{\frac{2n+1}{2}}}}\right),
\end{equation}
provided that $R=\frac{r}{\delta}$ is big enough. When $1/4\leq \epsilon\leq 1/2$ and $n\geq 2$,
\begin{eqnarray}
M_{1}:=\frac{4}{5}\abs{\cos(2\pi \epsilon-\frac{3\pi}{4})}-\frac{5}{4}{\sum_{k=2}^{+\infty}\frac{1}{k^{\frac{2n+1}{2}}}}\geq
\frac{4}{5}\cdot\frac{\sqrt{2}}{2}-\frac{5}{4}{\sum_{k=2}^{+\infty}\frac{1}{k^{\frac{5}{2}}}}\approx 0.138>0.\nonumber
\end{eqnarray}
Combining (\ref{eq:zhuanbian}), (\ref{eq:zhuanbian1}) and (\ref{eq:zhuanbian2}), we obtain the left side of (\ref{eq:thlowbup}).
Similarly, based on (\ref{eq:onthe}), we have
$$\abs{\sum_{k=1}^{+\infty}\frac{(-1)^{k}}{k^{n}}
J_{n}(2k\pi\frac{r}{\delta})}\,\,\leq\,\, \sum_{k=1}^{+\infty}\frac{1}{k^{n}}
\abs {J_{n}(2k\pi\frac{r}{\delta})}\,\,\leq\,\, \frac{5}{4}\frac{1}{\pi}\sqrt{\frac{\delta}{r}}\sum_{k=1}^{+\infty}\frac{1}{k^{\frac{2n+1}{2}}}
$$
provided that $R=\frac{r}{\delta}$ is big enough, which implies the right side of (\ref{eq:thlowbup}).

Now let us turn to
$\int_{0}^{\pi}\Delta_{\delta}(r\cos\theta)\cos\theta(\sin\theta)^{2n-1}d\theta$. Similar with the above,we have
 \begin{equation}\label{eq:mideq}
 \begin{aligned}
&\int_{0}^{\pi}\Delta_{\delta}(r\cos\theta)\cos\theta(\sin\theta)^{2n-1}d\theta\\ &=
-\frac{\delta}{\pi}\int_{0}^{\pi}
\sum_{k=1}^{+\infty}\frac{\sin(2k\pi\frac{r\cos\theta}{\delta}+k\pi)}{k}\cos\theta(\sin\theta)^{2n-1}d\theta\\
&= -\frac{\delta}{\pi}
\sum_{k=1}^{+\infty}\frac{(-1)^{k}}{k}\int_{0}^{\pi}\sin(2k\pi\frac{r}{\delta}\cos\theta)\cos\theta(\sin\theta)^{2n-1}d\theta\\
&= -\frac{2\delta}{\pi}
\sum_{k=1}^{+\infty}\frac{(-1)^{k}}{k}\sum_{m=0}^{+\infty}(-1)^{m}D_{m}J_{2m+1}(2k\pi\frac{r}{\delta})\\
&= -\frac{2\delta}{\pi}\sqrt{\pi}2^{n-\frac{3}{2}}(n-1)!\sum_{k=1}^{+\infty}\frac{(-1)^{k}}{k}
\frac{1}{(2k\pi\frac{r}{\delta})^{n-\frac{1}{2}}}J_{n+\frac{1}{2}}(2k\pi\frac{r}{\delta})\\
&=-\frac{(n-1)!}{\pi^{n}}\frac{\delta^{n+\frac{1}{2}}}{r^{n-\frac{1}{2}}}\sum_{k=1}^{+\infty} \frac{(-1)^k}{k^{n+\frac{1}{2}}}J_{n+\frac{1}{2}}(2k\pi \frac{r}{\delta}),
\end{aligned}
\end{equation}
where we use (\ref{eq:le211}) in Lemma \ref{le:2.2} for the fourth equality. Using the asymptotic estimate for $J_{n+\frac{1}{2}}(x)$ in Theorem \ref{le:Jvx}, similarly with the above, we can show that
\begin{equation}\label{eq:mideq1}
\frac{1}{\pi}\sqrt{\frac{\delta}{r}}\left(\frac{7}{8}|\cos(2\pi\epsilon-\frac{1}{2}\pi)|-\frac{8}{7}\sum_{k=2}^{+\infty}\frac{1}{k^{n+1}}\right)\leq|\sum_{k=1}^{+\infty} \frac{(-1)^k}{k^{n+\frac{1}{2}}}J_{n+\frac{1}{2}}(2k\pi \frac{r}{\delta})|\leq \frac{8}{7}\frac{1}{\pi}\sqrt{\frac{\delta}{r}}\sum_{k=1}^{+\infty}\frac{1}{k^{n+1}}
\end{equation}
provided that $R=\frac{r}{\delta}$ is big enough. When $1/6\leq \epsilon \leq 1/3$ and $n\geq 1$,
\begin{eqnarray}
M_{2}:=\frac{7}{8}|\cos(2\pi\epsilon-\frac{1}{2}\pi)|-\frac{8}{7}\sum_{k=2}^{+\infty}\frac{1}{k^{n+1}}\geq\frac{7}{8}\cdot\frac{\sqrt{3}}{2}-\frac{8}{7}
\sum_{k=2}^{+\infty}\frac{1}{k^{2}}\approx 0.02>0.\nonumber
\end{eqnarray}
Combing (\ref{eq:mideq}) and (\ref{eq:mideq1}), we arrive at (\ref{eq:thupbound}).
\end{proof}
We now can state the proof of the main theorem.
\begin{proof}[Proof of Theorem \ref{th:main}]
The idea to prove Theorem \ref{th:main} is similar to one of proving Theorem \ref{th:lowbound} in \cite{WX} with using
Lemma \ref{le:3.3} to estimate $\lim_{m\rightarrow \infty}E_\delta(x,\F_m)$. We state the
proof of (i) for the completeness. In fact, (ii) can be proved using a similar method.

We denote the
number of the non-zero entries in $x$ by $\|x\|_0$, i.e.,
$$
\|x\|_0 \,\,:=\,\,\#\{j:x_j\neq 0\}.
$$
 The proof is by induction on $\|x\|_0$. Note that
\begin{eqnarray*}
\lim_{m\rightarrow \infty}E_\delta(x,\F_m)
&=&\lim_{m\rightarrow
\infty}\|\frac{d}{N_m}\sum_{j=1}^{N_m}\Delta_\delta(x\cdot e_j
)e_j \|\\
&=&d\, \|\int_{{\bf z}\in \MS^d}\Delta_\delta(x\cdot {\bf z}
){\bf z}d\omega \|.
\end{eqnarray*}
 We begin with $\|x\|_0=1$.
 Without loss of generality, we suppose
$x=[x_1,0,\ldots,0]^T\in \R^{d}$ and consider
$\lim_{m\rightarrow \infty}E_\delta(x,\F_m)$. By the
sphere coordinate system, each ${z}=[z_1,\ldots,z_{d}]\in \MS^{d-1}$
can be written in the form of
\begin{eqnarray*}
[\cos\theta_1,\sin\theta_1\cos\theta_2,
\sin\theta_1\sin\theta_2\cos\theta_3,\ldots,\sin\theta_1\cdots
\sin\theta_{d-1}]^\top,
\end{eqnarray*}
where $\theta_1\in [0,\pi)$ and $\theta_j\in [-\pi,\pi), 2\leq
j\leq d-1$. To state conveniently, we set
$$
\Theta\,\,:=\,\, [0,\pi)\times \underbrace{[-\pi, \pi)\times
\cdots \times [-\pi,\pi)}_{d-2}, \qquad {S}_t(\theta):=\prod_{j=1}^{t}\sin\theta_j
$$
and
$$
J_d(\theta):=\left|(\sin\theta_1)^{d-2}(\sin\theta_2)^{d-3}\cdots
(\sin\theta_{d-2})\right|.
$$
 Noting that
$$
d\omega= { J}_d(\theta) d\theta_1\cdots d\theta_{d-1}\quad
\mbox{ and }\quad \int_{{z}\in \MS^{d-1}}\Delta_\delta(x_1  z_1){
z}_jd\omega=0,\quad\quad 2\leq j\leq d-1,
$$
 we have
\begin{eqnarray*}
\lim_{m\rightarrow \infty}E_\delta(x,\F_m)
&=&d\,\|\int_{{ z}\in \MS^{d-1}}\Delta_\delta(x\cdot
{z} ){z}d\omega \|\\
&=& d\,\left|\int_{{z}\in \MS^{d-1}}\Delta_\delta(x_1
{z}_1 ){z}_1d\omega \right|\\
&=& d\,\left|\int_{\theta\in
\Theta}\Delta_\delta(x_1\cos\theta_1)\cos\theta_1(\sin\theta_1)^{d-2}\right.\\
& &\qquad\left. |(\sin\theta_2)^{d-3}\cdots (\sin\theta_{d-2})| d \theta_1\cdots
d\theta_{d-1}\right|\\
&\geq & C_{1,d}\cdot\delta^{(d+1)/2}/|x_1|^{(d-1)/2}
\end{eqnarray*}
where the last inequality  follows from Lemma \ref{le:3.3}.

For the induction step, we suppose that the conclusion holds for
the case where $\|x\|_0\leq k$. We now consider $\|x\|_0 \leq
k+1$. Without loss of generality, we suppose $x$ is in the form of
$[0,\ldots,0,x_{d-k},\ldots,x_{d}]\in \R^{d}$. We can write
$[x_{d-1},x_{d}]$ in the form of $(r_0\cos\varphi_0,r_0\sin\varphi_0)$,
where $r_0\in\R_+$ and $\varphi_0\in [0,2\pi)$. Then
$$
x\cdot {z} =\sum_{j=d-k}^{d-2}x_jS_j(\theta)\cos\theta_j+
r_0\sin\theta_1\cdots \sin\theta_{d-2}\cos(\theta_{d-1}-\varphi_0)=:{
T}(\varphi_0).
$$
A simple observation is
\begin{eqnarray*}
& &\left(\int_{\theta\in\Theta}\Delta_\delta({ T}(\varphi_0)){
S}_{d-2}(\theta){ J}_d(\theta)\cos\theta_{d-1}
d\theta\right)^2+\left(\int_{\theta\in\Theta}\Delta_\delta({
T}(\varphi_0)){ S}_{d-2}(\theta){ J}_{d}(\theta) \sin\theta_{d-1}
d\theta\right)^2\\
&=&\left(\int_{\theta\in\Theta}\Delta_\delta({ T}(0)){
S}_{d-2}(\theta){ J}_d(\theta)\cos\theta_{d-1}
d\theta\right)^2+\left(\int_{\theta\in\Theta}\Delta_\delta({
T}(0)){ S}_{d-2}(\theta){ J}_d(\theta)\sin\theta_{d-1}
d\theta\right)^2.
\end{eqnarray*}
Then we have
\begin{eqnarray*}
& &\lim_{m\rightarrow \infty}E_\delta(x,\F_m)
=d\|\int_{{z}\in \MS^{d-1}}\Delta_\delta(x\cdot
{z} ){z}d\omega \|\\
&=&d\left(\sum_{t=d-k}^{d-1}(\int_{\theta\in\Theta}\Delta_\delta(T(\varphi_0))S_{t-1}(\theta)J_d(\theta)\cos\theta_t
d\theta)^2+(\int_{\theta\in\Theta}\Delta_\delta(T(\varphi_0))S_{d-1}(\theta)J_d(\theta)
d\theta)^2\right)^{1/2}\\
&=&d\left(\sum_{t=d-k}^{d-1}(\int_{\theta\in\Theta}\Delta_\delta(T(0))S_{t-1}(\theta)J_d(\theta)\cos\theta_t
d\theta)^2+(\int_{\theta\in\Theta}\Delta_\delta(T(0))S_{d-1}(\theta)J_d(\theta)
d\theta)^2\right)^{1/2}\\
&\geq & C_{1,d}\cdot \delta^{(d+1)/2}/r^{(d-1)/2}
\end{eqnarray*}
where the last inequality follows from the fact $\|x\|_0\leq k$
provided $\varphi_0=0$.
\end{proof}

\end{document}